\documentclass[a4paper,12pt]{amsart}
\usepackage{amsmath, amssymb, amsthm, xcolor,mathtools}
\usepackage{verbatim}
\usepackage{graphicx}
\usepackage{enumerate}
\usepackage[colorlinks=true, citecolor=blue, linkcolor=blue, bookmarks=true]{hyperref}

\newtheorem{theorem}{Theorem}
\newtheorem{lemma}[theorem]{Lemma}
\newtheorem{corollary}[theorem]{Corollary}
\newtheorem{proposition}[theorem]{Proposition}

\newtheorem*{theorem*}{Theorem}

\theoremstyle{definition}
\newtheorem{definition}[theorem]{Definition}
\newtheorem{remark}[theorem]{Remark}
\newtheorem{example}[theorem]{Example}
\newtheorem{notation}[theorem]{Notation}
\newtheorem*{question*}{Question}

\newcommand{\QT}{\mathrm{QT}_c}
\newcommand{\N}{\mathbb{N}}
\newcommand{\M}{\mathcal{M}}
\newcommand{\RR}{\mathcal{R}}
\newcommand{\R}{\mathbb{R}}
\newcommand{\C}{\mathbb{C}}
\newcommand{\K}{\mathbb{K}}

\newcommand{\Z}{\mathcal{Z}}

\newcommand{\tr}{\mathrm{tr}}
\newcommand{\eps}{\epsilon}

\newcommand{\completion}[2]{\overline{#1}{}^{#2}}

\allowdisplaybreaks
\numberwithin{equation}{section}

\title[Traces on the uniform tracial completion]{Traces on the uniform tracial completion of {$\mathcal{Z}$}-stable {C$^*$}-algebras}

\author[]{Samuel Evington}
\address{Mathematical Institute, University of M\"unster, Einsteinstrasse 62, 48149 M\"unster, Germany}
\email{evington@uni-muenster.de}

\thanks{Research partially supported by: Deutsche Forschungsgemeinschaft (DFG, German Research Foundation) – Project-ID 427320536 – SFB 1442; Germany's Excellence Strategy EXC 2044 390685587  Mathematics M{\"u}nster: Dynamics–Geometry–Structure;  ERC Advanced Grant 834267 - AMAREC}

\begin{document}

\begin{abstract}
    The uniform tracial completion of a C$^*$-algebra $A$ with compact trace space $T(A) \neq \emptyset$ is obtained by completing the unit ball with respect to the uniform 2-seminorm $\|a\|_{2,T(A)}=\sup_{\tau \in T(A)} \tau(a^*a)^{1/2}$.  The \emph{trace problem} asks whether every trace on the uniform tracial completion is the $\|\cdot\|_{2,T(A)}$-continuous extension of a trace on $A$.
    We answer this question positively in the case of C$^*$-algebras that tensorially absorb the Jiang--Su algebra, such as those studied in the Elliott classification programme.
\end{abstract}

\maketitle

\section{Introduction}
\renewcommand{\thetheorem}{\Alph{theorem}}

A trace on a unital operator algebra $A$ is a positive linear functional $\tau:A \rightarrow \C$ such that $\tau(1_A) = 1$ and $\tau(ab)=\tau(ba)$ for all $a,b \in A$.
The significance of traces for the structure and classification of operator algebras was already apparent in the foundational papers of Murray and von Neumann (\cite{MvN36,MvN37,MvN43}).
In essence, the fact that traces don't ``see'' the non-commutativity gives rise to numerical invariants from non-commutative structures.

The space $T(A)$ of all traces on a C$^*$-algebra $A$ is itself an invariant. 
Indeed, a major programme of recent research in C$^*$-algebras has been the Elliott classification programme (\cite{El76, El95}), which seeks to classify the simple separable amenable C$^*$-algebras via K-theory and traces under suitable regularity properties (see the survey articles \cite{WhiteICM, WinterICM, ET08} for an overview and \cite{Ki95, Phillips00,GLN20a, GLN20b, EGLN15, TWW17, CGSTW} for the state of the art). 

The motivating example of a trace is the trace on a matrix algebra (suitably normalised). 
In this case, the trace is unique. 
However, in general, the trace space $T(A)$ of a C$^*$-algebra $A$ can be empty, a singleton, a finite-dimensional simplex, or an infinite-dimensional simplex (in the sense of Choquet theory \cite{Choquet56, Alf71}). The last case can occur even for simple  approximately finite-dimensional C$^*$-algebras (\cite{Bl80, Go77}).

Traces on a C$^*$-algebra $A$ give rise to representations via the Gelfand--Naimark--Segal construction (\cite{GNS1,GNS2}). Indeed, given  $\tau \in T(A)$, one obtains a Hilbert space $H_\tau$, by completing $A$ with respect to the seminorm $\|a\|_{2,\tau} = \tau(a^*a)^{1/2}$, and a representation $\pi_\tau$ of $A$ on $H_\tau$ via left multiplication. Therefore, one obtains an enveloping von Neumann algebra $\pi_\tau(A)''$ for each trace $\tau \in T(A)$. 

When $A$ is amenable, the von Neumann algebra $\pi_\tau(A)''$ is hyperfinite by Connes' theorem (\cite{Co76}). A breakthrough of Matui and Sato established lifting techniques for deducing structural results about a simple nuclear C$^*$-algebra $A$ from properties of the von Neumann algebras $\pi_\tau(A)''$ when $T(A)$ was non-empty and finite-dimensional (\cite{MS12}). These ideas were subsequently extended to cover certain infinite-dimensional trace spaces (\cite{TWW15, KR14, Sa12}) and, when $A$ also has stable rank one and locally finite nuclear dimension, to all trace spaces (\cite{Thiel20}). 

The main difficulty when $T(A)$ is infinite-dimensional is that one now needs to work with the \emph{uniform 2-norm} $\|a\|_{2,T(A)} = \sup_{\tau \in T(A)}\|a\|_{2,\tau}$ and the infinitely many extreme points of $T(A)$ will no longer be a topologically discrete set.
Ozawa identified the \emph{uniform tracial completion} $\completion{A}{T(A)}$ of a C$^*$-algebra $A$ with $T(A) \neq \emptyset$ as the key object of study  (\cite{Oz13}). This C$^*$-algebra is obtained by completing the $\|\cdot\|$-closed unit ball of $A$ with respect to the norm $\|\cdot\|_{2,T(A)}$ (see Section \ref{subsec:completion}). 
Loosely speaking, $\completion{A}{T(A)}$ can be viewed as the section algebra of a bundle over $T(A)$ with the von Neumann algebras $\pi_\tau(A)''$ as fibres (\cite{Oz13, Ev16, Ev18}).
Structural properties of the uniform tracial completion underpin recent progress on the Elliott classification programme (\cite{CETW-classification, CGSTW}) and the Toms--Winter conjecture in particular (\cite{TW07, BBSTWW, CETWW, CE}). Moreover, they are the motivating example of tracially complete C$^*$-algebras (\cite{TraciallyComplete, TracialTransfer}).

This paper concerns a foundational issue in the theory of uniform tracial completions known as the trace problem (see \cite[Section 1.2]{TraciallyComplete}):

\begin{question*}[Trace Problem]
Is every trace on $\completion{A}{T(A)}$ the $\|\cdot\|_{2,T(A)}$-continuous extension of a trace on $A$?     
\end{question*}

At a more conceptual level, the trace problem is about whether the uniform tracial completion process is actually idempotent and whether the $\|\cdot\|_{2,T(A)}$-norm structure on $\completion{A}{T(A)}$ can be recovered from the C$^*$-algebraic structure of $\completion{A}{T(A)}$.
Some useful insight into the trace problem can be garnered from the monotracial case. When $T(A)=\{\tau\}$, it can be shown that $\completion{A}{T(A)}$ is isomorphic to the von Neumann algebra $\pi_\tau(A)''$ coming from the GNS construction. Therefore, the trace problem has a positive solution in this special case because $\pi_\tau(A)''$ is a finite factor and finite factors have a unique trace (\cite{MvN43}). 

The main result of this paper is a solution to the trace problem for arbitrary trace spaces in the setting relevant for the Elliott classification programme.  
\begin{theorem}\label{thm:main}
    Let $A$ be a C$^*$-algebra with $T(A)$ compact and non-empty. Suppose $A$ absorbs the Jiang--Su algebra $\Z$ tensorially, i.e.\ $A \otimes \Z \cong A$. Then the trace problem has a positive solution. 
\end{theorem}

The Jiang--Su algebra $\Z$ (\cite{JS99}), which appears in Theorem \ref{thm:main}, plays a fundamental role in the classification of simple nuclear C$^*$-algebras since both $A$ and $A \otimes \Z$ have the same K-theory and trace space. Tensorial absorption of $\Z$, also known as \emph{$\Z$-stability}, has emerged as the key regularity property in the Elliott classification programme (\cite{ET08}). In particular, all simple approximately finite-dimensional C$^*$-algebras are $\Z$-stable. 

The condition that $T(A)$ is compact is automatic when $A$ is unital. Theorem \ref{thm:main} is formulated so that it also covers a wide class of non-unital examples. 
In the non-unital case, the trace space $T(A)$ is defined as the space of positive tracial functionals $\tau \in A^*$ with $\|\tau\|_{A^*} = 1$ and it may or may not be compact. However, every simple separable exact $\Z$-stable C$^*$-algebra is stably isomorphic to one with compact trace space (see for example \cite[Theorem 2.7]{CE}).

The trace problem was previously known to have a positive solution for finite-dimensional trace simplices, where $\completion{A}{T(A)}$ is a finite direct sum of finite factors and so von Neumann algebraic methods suffice. Beyond this point, $\completion{A}{T(A)}$ is no longer a von Neumann algebra, so a solution to the trace problem has to take the topology of $T(A)$ and the bundle-like structure of $\completion{A}{T(A)}$ into account. Indeed, the trace problem was open even for approximately finite-dimensional algebras with a compact (but infinite) set of extreme traces. 

The solution to the trace problem provides further evidence for the central role that uniform tracial completions can play in the stably finite part of the Elliott classification programme and future equivariant or non-simple extensions thereof. 
At the conceptual level, it opens the door to analysis of the C$^*$-structure of $\completion{A}{T(A)}$ generalising von Neumann algebraic results. 
At the technical level, it simplifies arguments by circumventing the need to restrict to a space of $\|\cdot\|_{2,T(A)}$-continuous traces.

Beyond regularity, none of the other assumptions typically seen in the Elliott classification programme are needed to establish Theorem~\ref{thm:main}. In particular, it holds for non-simple C$^*$-algebras and non-nuclear C$^*$-algebras. Therefore, its future applications are not limited to the setting of the Elliott classification programme.
Moreover, $\Z$-stability only enters in order to ensure the existence of complemented partitions of unity (\cite{CETWW, CE}), which we will discuss below, and to rule out type I quotients. In particular, uniform property $\Gamma$ would be a suitable alternative regularity property (\cite{CETW}). 

For maximum generality, we work in the framework of tracially complete C$^*$-algebras (see Section \ref{subsec:TC-algebras} for the relevant definitions) and establish the following theorem:
\begin{theorem}\label{thm:main-TC}
    Let $(\M, X)$ be a type II$_1$ factorial tracially complete C$^*$-algebra with complemented partitions of unity. Then $T(\M) = X$.
\end{theorem}
Thus, the trace problem as stated in \cite[Question 1.1]{TraciallyComplete} has a positive solution in the presence of complemented partitions of unity (CPoU) in the type II$_1$ setting. Theorem~\ref{thm:main} is a special case of Theorem~\ref{thm:main-TC}, but we can also use Theorem~\ref{thm:main-TC} to compute the trace simplex of W$^*$-bundles with property $\Gamma$. An example application is the following:

\begin{corollary}\label{cor:main-bundle}
    Let $K$ be a compact Hausdorff space and $\RR$ denote the hyperfinite II$_1$ factor. Let $\M = C_\sigma(K, \RR)$ be the trivial W$^*$-bundle over $K$ with fibre $\RR$, i.e.\  $\M$ is the C$^*$-algebra of all $\|\cdot\|$-bounded and $\|\cdot\|_{2,\mathrm{tr}_\RR}$-continuous functions $f:K \rightarrow \RR$.
    Then every trace $\tau \in T(\M)$ has the form 
    \begin{equation}\label{eqn:main-bundle1}
        \tau(f) = \int_K \mathrm{tr}_\RR(f(x)) \, d\mu(x), \quad f \in \M,
    \end{equation}
    for some Radon probability measure $\mu \in \mathrm{Prob}(K)$.
\end{corollary}

Theorem~\ref{thm:main-TC} is proven using the theory of Cuntz subequivalence  (\cite{Cu78}). The strategy is inspired by Murray and von Neumann's original proof that II$_1$ factors have a unique trace: they were able to prove enough about Murray--von Neumann subequivalence of projections that there was only one possible candidate for a tracial state (\cite{MvN36,MvN37}). In our case, we prove the following result (see Section \ref{subsec:strict-comparison} for the relevant definitions):
\begin{theorem}\label{thm:main-strict-comparison}
    Let $(\M, X)$ be a type II$_1$ factorial tracially complete C$^*$-algebra with complemented partitions of unity. Then $\M$ has strict comparison with respect to the traces in $X$.
\end{theorem}
Unlike Murray and von Neumann, we need to work with a comparison theory for general positive elements (not just projections) because we don't know that projections have a $\|\cdot\|$-dense linear span. 
The argument for deducing Theorem~\ref{thm:main-TC} from Theorem~\ref{thm:main-strict-comparison} is based on a result of Ng--Robert (\cite{NgR16}) together with some Choquet theory.
The idea is that the traces in $X$ determine the Cuntz comparison theory of $\M$ (up to a small error) thanks to strict comparison, which in turn imposes such severe constraints on all traces of $\M$ that $T(\M) = X$ holds.

It remains to say a few words about the proof of Theorem~\ref{thm:main-strict-comparison}. The main idea is to show that hereditary C$^*$-subalgebras of $\M$ contain a wealth of projections. In essence, we show that $\M$ has real rank zero at the level of the Cuntz semigroup; see Theorem~\ref{thm:Cuntz-dense} for the formal statement. This combined with the Murray--von Neumann comparison theory for projections developed in \cite{TraciallyComplete} is used to prove Theorem~\ref{thm:main-strict-comparison}.

It is important to emphasise that we are working with $\|\cdot\|$-closed hereditary C$^*$-subalgebras and that the Cuntz subequivalence also uses the $\|\cdot\|$-norm (as usual). Indeed, if we only established analytic properties of $\M$ with respect to the uniform 2-norm $\|\cdot\|_{2,X}$, we could not expect to prove anything about non-$\|\cdot\|_{2,X}$-continuous traces. 

The projections in $\M$ are obtained as limits of $\|\cdot\|$-bounded, $\|\cdot\|_{2,X}$-Cauchy sequences of approximate projections. By completeness, such sequences converge in $\M$. We are able to construct these Cauchy sequences in a fixed $\|\cdot\|$-closed hereditary C$^*$-subalgebra and show that the limit remains in a (slightly larger) $\|\cdot\|$-closed hereditary C$^*$-subalgebra; see Theorem \ref{thm:projections}.

The heavy lifting takes place in Lemma~\ref{lem:onestep}, where we construct approximate projections in a $\|\cdot\|$-closed hereditary C$^*$-subalgebra of $\M$, while remaining close to a previously constructed approximate projection. The proof uses complemented partitions of unity (CPoU) in order to ``glue together'' the analogous von Neumann algebraic result (see Proposition~\ref{prop:close-projection}), which holds in each fibre $\pi_\tau(\M)''$. CPoU (see Definition~\ref{def:CPoU}) was developed in \cite{CETWW} precisely for implementing this kind of ``tracial gluing'' argument. A good introduction to this technique can be found in \cite[Section 7]{TraciallyComplete}.

Experts will be aware that CPoU is usually only able to prove facts about a tracially complete C$^*$-algebra $(\M, X)$ up to a small $\|\cdot\|_{2,X}$-error, which would not suffice for this application. The crucial new observation in this paper is that $\|\cdot\|$-closed hereditary C$^*$-subalgebras of $\M$ can be used to provide some level of $\|\cdot\|$-norm control over the output of CPoU arguments. This observation is likely to have further applications to the structure theory of tracially complete C$^*$-algebras beyond the trace problem.

\subsection*{Acknowledgements}
I'd like to thank Aaron Tikuisis, Hannes Thiel, Andrea Vaccaro and Stuart White for their comments on an earlier version of this manuscript. I'd also like to thank the anonymous referee for their comments.

\renewcommand{\thetheorem}{\arabic{theorem}}
\numberwithin{theorem}{section} 
\section{Preliminaries}\label{sec:prelims}

In this section, we recall the key definitions used in this paper and collect the required preliminaries for the main argument.

\subsection{Traces and Choquet simplices}\label{subsec:choquet}
By a \emph{trace} on a C$^*$-algebra $A$, we mean a tracial state. We write $T(A)$ for the set of all tracial states on $A$. The trace space $T(A)$ is a convex subset of $A^*$. We endow $T(A)$ with the subspace topology induced by the weak$^*$ topology on $A^*$. 

A \emph{Choquet simplex} $X$ is a compact convex set where every $x \in X$ is the barycentre of a unique Radon probability measure $\mu_x$ supported on the extreme boundary $\partial_e X$ of $X$; see for example \cite[Theorem II.3.6]{Alf71}.\footnote{When $X$ is non-metrizable, $\partial_e X$ need not be Borel, so \emph{supported on the extreme boundary} should be interpreted as $\mu_x(E) = 0$ for any Baire measurable set $E$ that doesn't intersect $\partial_e X$.} 
The trace space $T(A)$ of a unital C$^*$-algebra is a Choquet simplex by \cite[Theorem~3.1.18]{Sak98}. 
When $A$ is non-unital, $T(A)$ need not be compact; however, if $T(A)$ is compact, then $T(A)$ is a Choquet simplex; see for example \cite[Theorem 2.6]{TraciallyComplete}. 

In this subsection, we recall some important facts about continuous affine function on Choquet simplices in general and on $T(A)$ in particular. We write $\mathrm{Aff}(X)$ for the space of continuous affine functions $X \rightarrow \R$, $\mathrm{Aff}(X)_+$ for the non-negative valued affine functions $X \rightarrow [0,\infty)$, and $\|\cdot\|_\infty$ for the supremum norm on $\mathrm{Aff}(X)$.
The following result is well-known and is often attributed to Cuntz and Pedersen (\cite{CP79}); see \cite[Proposition~2.1]{CGSTW} and \cite[Proposition~2.7]{TraciallyComplete} for a proof.
\begin{proposition}
	\label{prop:CP}
	Let $A$ be a unital C$^*$-algebra and let $f \in \mathrm{Aff}(T(A))$ be a continuous affine function.
	Then for any $\epsilon>0$, there is a self-adjoint element $a \in A$ such that
	\begin{equation} \|a\| \leq \|f\|_\infty + \epsilon \quad \text{and} \quad \tau(a) = f(\tau), \quad \tau \in T(A). \end{equation}
        Moreover, if $f(\tau) > 0$ for all $\tau \in T(A)$, we may assume $a \in A_+$.
\end{proposition}

The other results in this subsection concern closed faces of Choquet simplices. Recall that a \emph{face} of a convex set $X$ is a convex subset $F \subseteq X$ such that for all $\tau_1,\tau_2 \in X$, we have $\tau_1, \tau_2 \in F$ whenever $\tfrac{1}{2}(\tau_1+\tau_2) \in F$. 

Firstly, we record the following extension theorem for continuous affine functions.
\begin{proposition}[{\cite[Theorem~II.5.19]{Alf71}}]\label{prop:extending-affine}
	Let $F$ be a closed face of a Choquet simplex $X$.
	For every $f \in \mathrm{Aff}(F)$, there exists $\hat{f} \in \mathrm{Aff}(X)$ with $\hat{f}|_F = f$ and $\|\hat{f}\|_\infty = \|f\|_\infty$. 
\end{proposition}

Secondly, we note that every closed face of a Choquet simplex is \emph{relatively exposed}, in the sense of the following proposition.
\begin{proposition}[{\cite[Corollary~II.5.20]{Alf71}}]\label{prop:exposed}
	Let $F$ be a closed face in the Choquet simplex $X$, and let $x_0 \in X \setminus F$.
        Then there exists $f \in \mathrm{Aff}(X)_+$ with $f|_F = 0$ and $f(x_0) > 0$.
\end{proposition}

\subsection{Uniform trace norms and uniform tracial completions}\label{subsec:completion}

The uniform tracial completion, also known as the strict closure, of a C$^*$-algebra $A$ with $T(A) \neq \emptyset$ was introduced by Ozawa in \cite{Oz13}. In this subsection, we recall the construction and set out our notational conventions. We work in a slightly more general setting than in the introduction to this paper, allowing uniform trace norms with respect to subsets of the trace simplex. General references for the material in this subsection are \cite{Oz13, Ev16, TraciallyComplete}.

Let $A$ be a C$^*$-algebra. Given a non-empty compact convex subset $X \subseteq T(A)$, we define the seminorm
\begin{equation}
    \|a\|_{2,X} = \sup_{\tau \in X} \|a\|_{2,\tau} = \sup_{\tau \in X} \tau(a^*a)^{1/2}
\end{equation}
for $a \in A$. If this seminorm is a norm, we say that $X$ is \emph{faithful} and call $\|\cdot\|_{2,X}$ the \emph{uniform 2-norm} with respect to $X$. It is easily seen that addition, scalar multiplication, the adjoint, and every $\tau \in X$ is continuous with respect to $\|\cdot\|_{2,X}$. Multiplication is  $\|\cdot\|_{2,X}$-continuous when restricted to a $\|\cdot\|$-bounded subset of $A$, since 
\begin{equation}
    \|ab\|_{2,X} \leq \|a\|\|b\|_{2,X}
\end{equation}
for all $a,b \in A$. This inequality also shows the $\|\cdot\|_{2,X}$-continuity of left multiplication by a fixed $a \in A$. Taking adjoints, we get that $\|ab\|_{2,X} \leq \|a\|_{2,X}\|b\|$ for all $a,b \in A$; the $\|\cdot\|_{2,X}$-continuity of right multiplication follows. We can now define the uniform tracial completion.

\begin{definition}\label{def:UTCompletion}
	Let $A$ be a C$^*$-algebra. For a compact convex set $X \subseteq T(A)$, the \emph{uniform tracial completion} of $A$ with respect to $X$ is the C$^*$-algebra
	\begin{equation}
		\completion{A}{X} = \frac{\{(a_n)_{n=1}^\infty \in \ell^\infty(A): (a_n)_{n=1}^\infty \text{ is }\|\cdot\|_{2,X}\text{-Cauchy}\}}{\{(a_n)_{n=1}^\infty \in \ell^\infty(A): (a_n)_{n=1}^\infty \text{ is }\|\cdot\|_{2,X}\text{-null}\}}.
	\end{equation}
        The case $X=T(A)$ is of particular interest.
\end{definition}

The $^*$-homomorphism $\iota:A \to \completion{A}{X}$ given by sending $a \in A$ to the image of the constant sequence $(a,a,\dots)$ is an embedding when $X$ is faithful, and by replacing $A$ with a suitable quotient there is no loss of generality by restricting to this case. We shall identify $A$ with $\iota(A)$. 

The uniform 2-norm on $A$ extends to a norm on the uniform tracial completion given by $(a_n)_{n=1}^\infty \mapsto \lim_{n\rightarrow\infty} \|a_n\|_{2,X}$. 
The standard diagonal arguments show that the $\|\cdot\|$-closed unit ball of $\completion{A}{X}$ is $\|\cdot\|_{2,X}$-complete and that the $\|\cdot\|$-closed unit ball of $A$ is $\|\cdot\|_{2,X}$-dense in the $\|\cdot\|$-closed unit ball of $\completion{A}{X}$; see for example \cite[Proposition 3.23]{TraciallyComplete}.

Every $\tau \in X$ has a unique $\|\cdot\|_{2,X}$-continuous extension given by  $(a_n)_{n=1}^\infty \mapsto \lim_{n\rightarrow\infty} \tau(a_n)$. Hence, we identify $X$ with a subset of $T(\completion{A}{X})$. The trace problem asks whether equality holds (for $X = T(A)$). Theorem \ref{thm:main} answers this question positively in the case that $A$ is $\Z$-stable.

\subsection{Tracially complete \texorpdfstring{C$^*$}{C*}-algebras}\label{subsec:TC-algebras}

In this paper, we work in the general framework of \emph{tracially complete C$^*$-algebras}, introduced in \cite{TraciallyComplete} and further investigated in \cite{TracialTransfer}. This framework includes as special cases both the uniform tracial completions of C$^*$-algebras discussed above and Ozawa's W$^*$-bundles, introduced in \cite{Oz13}.

\begin{definition}\label{def:TC}
	A \emph{tracially complete C$^*$-algebra} is a pair $(\M,X)$ where $\M$ is a unital C$^*$-algebra and $X \subseteq T(\M)$ is a compact convex set such that
	\begin{enumerate}
		\item $\|\cdot\|_{2,X}$ is a norm on $\M$, and
		\item the $\|\cdot\|$-closed unit ball of $\M$ is $\|\cdot\|_{2,X}$-complete.
	\end{enumerate}
\end{definition}

A tracially complete C$^*$-algebra $(\M, X)$ is said to be \emph{factorial} if $X$ is a face in $T(\M)$. 
In this case, $X$ will be a Choquet simplex, since $T(\M)$ is a Choquet simplex.
Examples of factorial tracially complete C$^*$-algebras include uniform tracial completions $(\completion{A}{X}, X)$, when $X \subseteq T(A)$ is a compact face, and W$^*$-bundles with factorial fibres.
A tracially complete C$^*$-algebras $(\M, X)$ is said to be \emph{type \rm{II}$_1$} if $\pi_\tau(\M)''$ is a type II$_1$ von Neumann algebra for all $\tau \in X$.
For further information, see \cite[Section 3]{TraciallyComplete}.

Given a free ultrafilter $\omega \in \beta\N \setminus \N$, the \emph{ultrapower} $(\M^\omega,X^\omega)$ of the tracially complete C$^*$-algebra $(\M,X)$ is defined as follows. First, we set
\begin{equation}
    \M^\omega = \frac{\ell^\infty(\M)}{\{(a_n)_{n=1}^\infty \in \ell^\infty(A): \lim_{n\to\omega}\|a_n\|_{2,X} = 0\}}.
\end{equation}
Then, for every sequence of traces $(\tau_n)_{n=1}^\infty$ in $X$, we define a \emph{limit trace} on $\M^\omega$ via $(a_n)_{n=1}^\infty \mapsto \lim_{n\to\omega} \tau_n(a_n)$, and we set $X^\omega \subseteq T(\M)$ to be the weak$^*$-closure of the set of all limit traces. For further information, see \cite[Section 5.1]{TraciallyComplete}.  

The key technical machinery used in this paper is complemented partitions of unity (CPoU). Informally, this allows us to prove results about a tracially complete C$^*$-algebra $(\M,X)$ by gluing together results that are known to hold in the finite von Neumann algebras $\pi_\tau(\M)''$ for $\tau \in X$. Complemented partitions of unity were first introduced in \cite{CETWW}. The definition below is taken from \cite{TraciallyComplete}.

\begin{definition}[{\cite[Definition~6.1]{TraciallyComplete}}]\label{def:CPoU}
	Let $(\M,X)$ be a factorial tracially complete C$^*$-algebra. We say that $(\M,X)$ has \emph{complemented partitions of unity} (CPoU) if for any $\|\cdot\|_{2,X}$-separable subset $S \subseteq \M$, any family $a_1,\dots,a_k$ of positive elements in $\M$, and any scalar
	\begin{equation}\label{eq:CPoUTraceIneq1}
		\delta>\sup_{\tau \in X} \min_{1 \leq i \leq k}\tau(a_i),
	\end{equation}
	there exist orthogonal projections $q_1,\dots,q_k\in \M^\omega\cap S'$ summing to $1_{\M^\omega}$ such that
	\begin{equation}\label{eq:CPoUTraceIneq2}
		\tau(a_iq_i)\leq \delta\tau(q_i)
	\end{equation}
        for all $\tau\in X^\omega$ and $i=1,\dots,k$.
\end{definition}

The uniform tracial completion of a $\Z$-stable C$^*$-algebra has CPoU by \cite[Theorem 1.4]{TraciallyComplete}. For separable nuclear $\Z$-stable C$^*$-algebras, this result was originally shown in \cite[Theorem I]{CETWW}.

\subsection{Finite von Neumann algebras}\label{subsec:finite-vna}

In this subsection, we record a couple of technical lemmas about finite von Neumann algebras. These lemmas are the fibrewise results that we will glue together using CPoU in order to establish our main theorems. 

The first lemma is 2-norm stability of projections. A proof of this lemma can be found in \cite{Tak03}.
\begin{lemma}[{\cite[Chapter XIV, Lemma 2.2]{Tak03}}]\label{lem:stability}
Let $M$ be von Neumann algebra and $\tau \in T(M)$ normal. Let $0 < \epsilon < \tfrac{1}{4}$. Suppose $e \in M_{+,1}$ satisfies 
\begin{equation}
     \|e^2-e\|_{2,\tau} < \eps.
\end{equation}
Then $p = \chi_{[\sqrt{\eps}, 1]}(e)$ is a projection and $\|e-p\|_{2,\tau} < 2\sqrt{\eps}$.
\end{lemma}
Note that the projection $p$ constructed in Lemma~\ref{lem:stability} does not depend on the trace $\tau$, so $\|\cdot\|_{2,T(M)}$-stability of projections in a finite von Neumann algebra $M$ is an immediate consequence of Lemma~\ref{lem:stability} and the fact that the normal traces are weak$^*$ dense in the trace space.

The second lemma constructs a projection $p$ with specified tracial behaviour that is as close as possible to a given projection $q_0$.
This is a standard application of basic properties of type II$_1$ von Neumann algebras, and is likely known to experts, but we supply a proof for completeness.  

\begin{lemma}\label{lem:closest}
Let $M$ be a type II$_1$ von Neumann algebra. Let $q_0, q_1 \in M$ be projections with $q_0 \leq q_1$.
Let $f:T(M) \rightarrow [0,1]$ be a continuous affine function such that $f(\tau) \leq \tau(q_1)$ for all $\tau \in T(M)$. 

There exists a projection $p \in M$ such that 
\begin{enumerate}
\item $p \leq q_1$,
\item $\tau(p) = f(\tau)$ for all $\tau \in T(M)$,
\item $\|p-q_0\|^2_{2,T(M)} \leq \sup_{\tau \in T(M)} |\tau(q_0) - f(\tau)|$.
\end{enumerate}
\end{lemma}
\begin{proof}
    As $M$ is type II$_1$, there is a projection $p'' \in M$ with $\tau(p'') = f(\tau)$ for all $\tau \in T(M)$; see for example \cite[Proposition 2.8(iv)]{TraciallyComplete}.
    Since $f(\tau) \leq \tau(q_1)$ for all $\tau \in T(M)$, there exists a projection $p' \in M$ that is Murray--von Neumann equivalent to $p''$ and such that
    $p' \leq q_1$. 
    
    By generalised comparison of projections in a von Neumann algebra (see for example \cite[Proposition III.1.1.10]{Bl06}), there is a central projection $z \in Z(M)$ such that $zp' \lesssim zq_0$ and $(1-z)q_0 \lesssim (1-z)p'$. 
    Moreover, since $p', q_0 \leq q_1$, these Murray--von Neumann subequivalences can be realised by partial isometries in $q_1Mq_1$.
    Hence, there exists a projection $p_z \in zM$ that is Murray--von Neumann equivalent to $zp'$ and satisfies $p_z \leq zq_0 \leq zq_1$.
    Similarly, there exists a projection $p_{1-z} \in (1-z)M$ that is Murray--von Neumann equivalent to $(1-z)p'$ and that satisfies $(1-z)q_0 \leq p_{1-z} \leq (1-z)q_1$.

    As $z$ is a central projection, $p_z$ and $p_{1-z}$ are orthogonal projections.
    Therefore, the sum $p = p_z + p_{1-z}$ is a projection in $M$. As $p_z, p_{1-z} \leq q_1$, we have $p \leq q_1$.
    Since $z(q_0 - p_z)$ and $(1-z)(p_{1-z} - q_0)$ are orthogonal projections, we have 
    \begin{equation}\label{eqn:abs-squared}
    |p-q_0|^2 = z(q_0 - p_z) + (1-z)(p_{1-z} - q_0).
    \end{equation}
     
    Let $\alpha = \sup_{\tau \in T(M)} |\tau(q_0) - f(\tau)|$. 
    Since $M = zM \oplus (1-z) M$, every trace on $M$ is a convex combination of a trace supported on $zM$ and a trace support on $(1-z)M$. Hence, it suffices to show that $\|p-q_0\|_{2,\tau}^2 \leq \alpha$ for all traces supported on $zM$ and for all traces supported on $(1-z) M$.

    Suppose $\tau \in T(M)$ is supported on $zM$. Then by \eqref{eqn:abs-squared}, we have
    \begin{equation}
         \|p-q_0\|_{2,\tau}^2 = \tau(|p-q_0|^2) = \tau(q_0 - p_z) = \tau(q_0) - \tau(p_z).
    \end{equation}
    Since $p_z$ is equivalent to $zp'$ and $\tau$ is supported on $zM$, we have $\tau(p_z) = \tau(zp') = \tau(p') = f(\tau)$. Hence, $\|p-q_0\|_{2,\tau}^2 \leq \alpha$.
    
    Suppose $\tau \in T(M)$ is supported on $(1-z)M$. Then 
    \begin{equation}
         \|p-q_0\|_{2,\tau}^2 = \tau(p_{1-z}) - \tau(q_0).
    \end{equation} 
    by \eqref{eqn:abs-squared}.
    Since $p_{1-z}$ is equivalent to $(1-z)p'$ and $\tau$ is supported on $(1-z)M$, we have $\tau(p_{1-z}) = \tau((1-z)p') = \tau(p') = f(\tau)$. Hence, $\|p-q_0\|_{2,\tau}^2 \leq \alpha$.
    This completes the proof.
\end{proof}

Combining the previous two results, we get the following proposition.
\begin{proposition}\label{prop:close-projection}
Let $M$ be a type $II_1$ von Neumann algebra. Let $q_1 \in M$ be a projection. Let $f:T(M) \rightarrow [0,1]$ be a continuous affine function such that $f(\tau) \leq \tau(q_1)$ for all $\tau \in T(M)$. Let $0 < \eps < \tfrac{1}{4}$. Let $e \in q_1Mq_1$ be a positive contraction. Suppose 
\begin{align}
	\|e^2-e\|_{2,T(M)} &< \eps\\
	\sup_{\tau \in T(M)}|\tau(e)-f(\tau)| &< \eps. 
\end{align}
Then there exists a projection $p \in q_1Mq_1$ with $\|e-p\|_{2,T(M)} < 4\eps^{1/4}$ and $\tau(p) = f(\tau)$ for all $\tau \in T(M)$.
\end{proposition}
\begin{proof}
The normal traces are weak$^*$ dense in $T(M)$; see for example \cite[Proposition 2.8(ii)]{TraciallyComplete}. Hence, the projection $q_0 = \chi_{[\sqrt{\eps}, 1]}(e) \in q_1Mq_1$ satisfies $\|e - q_0\|_{2,T(M)} \leq 2\sqrt{\eps}$ by Lemma \ref{lem:stability}.
By the triangle inequality, we have that $\sup_{\tau \in T(M)} |\tau(q_0)-f(\tau)| < \eps + 2\sqrt{\eps}$.
By Lemma \ref{lem:closest}, there exists a projection $p \in q_1Mq_1$ with $\tau(p) = f(\tau)$ for all $\tau \in T(M)$ such that
$\|p-q_0\|_{2,T(M)}^2 < \eps + 2\sqrt{\eps}$.
By the triangle inequality, we have
\begin{equation}
\begin{split}
	\|p-e\|_{2,T(M)} &\leq \|p-q_0\|_{2,T(M)}  + \|e-q_0\|_{2,T(M)}\\
	&< \sqrt{\eps + 2\sqrt{\eps}} +2\sqrt{\eps} \\
	&< 4\eps^{1/4},
\end{split}
\end{equation}
as $0 < \eps < \tfrac{1}{4}$. 
\end{proof}

\begin{remark}
    Proposition~\ref{prop:close-projection} is still true when $q_1 = 0$. The function $f$ is forced to be identically zero and $p$ is the zero projection. This degenerate case can occur in the main complemented partition of unity argument (Lemma~\ref{lem:onestep}).
\end{remark}

\subsection{Strict comparison}\label{subsec:strict-comparison}

In this subsection, we record the results on strict comparison that we will need in this paper. Our main references are \cite{ERS11} and \cite{NgR16}.

For a C$^*$-algebra $A$, we write $\QT(A)$ for the cone of all lower semi-continuous 2-quasitraces $\tau:A_+ \rightarrow [0,\infty]$. Every lower semi-continuous 2-quasitrace on $A$ extends uniquely to $A \otimes \K$. In this paper, we shall abbreviate lower semi-continuous 2-quasitrace to \emph{quasitrace}.  
The topology on $\QT(A)$ is specified by defining the convergent nets. 
A net $(\tau_\lambda)$ in $\QT(A)$ converses to $\tau$ if and only if 
\begin{equation}
	\limsup_\lambda \tau_\lambda((a-\eps)_+) \leq \tau(a) \leq \liminf_\lambda \tau_\lambda(a)
\end{equation}
for all $a \in (A \otimes \K)_+$ and $\eps > 0$.
Endowed with this topology, $\QT(A)$ is a compact Hausdorff cone; see \cite[Section 4.1]{ERS11}. 
The subspace topology on $T(A)$ inherited from $\QT(A)$ and the subspace topology on $T(A)$ induced by the weak$^*$-topology on $A^*$ coincide; see \cite[Proposition 3.10]{ERS11}.\footnote{We warn the reader that \cite{ERS11} uses a different notational convention: $T(A)$ denotes the cone of extended traces and the bounded traces are $T_I(A)$ where $I=A$.} 
In particular, $T(A)$ is a compact subset of $\QT(A)$ when $A$ is unital. 

For every $\tau \in \QT(A)$, we write $d_\tau:(A\otimes\K)_+ \rightarrow [0,\infty]$ for the \emph{rank function} given by $a \mapsto \lim_{n\to\infty} \tau(a^{1/n})$. 
For $a,b \in (A\otimes\K)_+$, we write $a \precsim b$ for \emph{Cuntz subequivalence}, i.e.\ if there exists a sequence $(r_n)_{n=1}^\infty$ in  $A\otimes\K$ such that $\lim_{n\to\infty} r_nbr_n^* = a$. 
It is standard that $a \precsim b$ implies $d_\tau(a) \leq d_\tau(b)$ for all $\tau \in \QT(A)$. Loosely speaking, an algebra has \emph{strict comparison} when a partial converse to this implication holds.

In this paper, we shall use the following notion of strict comparison with respect to a compact subset of $\QT(A)$ due to Ng--Robert. The case $K = \QT(A)$ corresponds to strict comparison.
\begin{definition}[{\cite[Definition 3.1]{NgR16}}]\label{def:strict-comparison-K}
	Let $A$ be a C$^*$-algebra and $K \subseteq \mathrm{QT}_c(A)$ be a compact subset. We say that $A$ has \emph{strict comparison with respect to $K$} if, for any $a,b \in (A \otimes \K)_+$, we have $a \precsim b$ whenever there exists $\gamma > 0$ such that, for all $\tau \in K$,  $d_\tau(a) \leq (1-\gamma)d_\tau(b)$.
\end{definition}

Ng and Robert show that having strict comparison with respect to a compact subset of $\QT(A)$ imposes quite severe restrictions on the remaining quasitraces. We restate their lemma for the benefit of the reader.
\begin{lemma}[{\cite[Lemma 3.4]{NgR16}}] \label{lem:K-to-all}
	Let $A$ be a C$^*$-algebra that has strict comparison with respect to the compact set $K \subseteq \mathrm{QT}_c(A)$. Let $a,b \in (A \otimes \K)_+$.
\begin{enumerate}
	\item If $d_\tau(a) \leq d_\tau(b)$ for all $\tau \in K$, then $d_\tau(a) \leq d_\tau(b)$ for all $\tau \in \QT(A)$.
	\item If $\tau(a) \leq \tau(b)$ for all $\tau \in K$, then $\tau(a) \leq \tau(b)$ for all $\tau \in \QT(A)$.
\end{enumerate}
\end{lemma} 
This result has recently been put into an abstract framework where it can be viewed as an application of a Hahn--Banach separation theorem; see \cite[Appendix A]{RFRT24}. 

The following technical lemma, also due to Ng and Robert, allows us to take advantage of cut-down arguments when it comes to proving strict comparison.
\begin{lemma}[{\cite[Proposition 3.3]{NgR16}}] \label{lem:cut-down}
	Let $A$ be a C$^*$-algebra and $K \subseteq \mathrm{QT}_c(A)$ be a compact subset. Let $a,b \in (A \otimes \K)_+$ and $\gamma > 0$. Suppose that $d_\tau(a) \leq (1-\gamma)d_\tau(b)$ for all $\tau \in K$. Then for every $\eps > 0$ there exists $\delta > 0$ such that
\begin{equation}
	d_\tau((a-\eps)_+) \leq \left(1-\frac{\gamma}{2}\right)d_\tau((b-\delta)_+)
\end{equation}
for all $\tau \in K$.
\end{lemma}

Using the previous lemma and standard stability results for Cuntz subequivalence, we can show that it suffices to consider $a, b \in A \otimes M_n$ in the definition of strict comparison. This result will be convenient when it comes to considering tracially complete C$^*$-algebras, as this class of algebras is closed under matrix amplifications but not under stabilisation. 
\begin{lemma}\label{lem:only-matrix}
		Let $A$ be a C$^*$-algebra and $K \subseteq \mathrm{QT}_c(A)$ be a compact subset. Suppose that for any $n \in \N$, $\gamma > 0$, and $a,b \in (A \otimes M_n)_+$ such that $d_\tau(a) \leq (1-\gamma)d_\tau(b)$ for all $\tau \in K$, we have $a \precsim b$. Then $A$ has strict comparison with respect to $K$.
\end{lemma} 
\begin{proof}
        View $M_n \subseteq \K$ in the standard way and write $1_n$ for the unit of $M_n$. 
	Let $(e'_{\lambda'})_{\lambda' \in \Lambda'}$ be an increasing approximate unit for $A$. 
        Let $\Lambda = \Lambda' \times \N$ be the product of the directed sets.
	Define an increasing approximate unit $(e_\lambda)_{\lambda \in \Lambda}$ of $A \otimes \K$, by setting $e_\lambda = e'_{\lambda'} \otimes 1_n$ for each $\lambda = (\lambda',n) \in \Lambda$.
 
	Let $a,b \in (A \otimes \K)_+$ and $\gamma > 0$.
	Suppose that $d_\tau(a) \leq (1-\gamma)d_\tau(b)$ for all $\tau \in K$.
	Fix $\eps > 0$. By Lemma \ref{lem:cut-down}, there is $\delta > 0$ such that 
        \begin{equation}
            d_\tau((a-\eps)_+)\leq \left(1-\frac{\gamma}{2}\right)d_\tau((b-\delta)_+) 
        \end{equation} 
        for all $\tau \in K$. Choose $\lambda \in \Lambda$ such that $\|(a-\eps)_+-e_\lambda(a-\eps)_+e_\lambda\| < \eps$ and $\|b - e_\lambda b e_\lambda\| < \delta$.
	Then, by \cite[Proposition 2.2]{Ro92}, we have
        \begin{align}
                (a-2\eps)_+ &\precsim e_\lambda(a-\eps)_+e_\lambda \precsim (a-\eps)_+ \mbox{ and } \label{eqn:Rordam-1}\\
                (b-\delta)_+ &\precsim e_\lambda b e_\lambda \precsim b. \label{eqn:Rordam-2}
        \end{align}   
	Say $\lambda = (\lambda',n)$ where $\lambda' \in \Lambda'$ and $n \in \N$. Then both $e_\lambda(a-\eps)_+e_\lambda$ and $e_\lambda b e_\lambda$ are elements of $(A \otimes M_n)_+$, and we have
	\begin{equation}
	\begin{split}
	    d_\tau(e_\lambda (a-\eps)_+e_\lambda) &\leq d_\tau((a-\eps)_+) \\
		&\leq \left(1-\frac{\gamma}{2}\right)d_\tau((b-\delta)_+) \\
		&\leq \left(1-\frac{\gamma}{2}\right)d_\tau(e_\lambda b e_\lambda)
	\end{split}
	\end{equation}
	for all $\tau \in K$.
	Therefore, by our hypothesis, $e_\lambda (a-\eps)_+e_\lambda \precsim e_\lambda b e_\lambda$. 
	Hence, by \eqref{eqn:Rordam-1} and \eqref{eqn:Rordam-2}, $(a-2\eps)_+ \precsim b$.
	Since $\eps$ was arbitrary, $a \precsim b$.
\end{proof}

\section{The main construction}\label{sec:strict-comparison-TP}

We now carry out the strategy discussed in the introduction. The first technical result is the construction of approximate projections in hereditary subalgebras of a tracially complete C$^*$-algebra with complemented partitions of unity (CPoU). 

\begin{lemma}\label{lem:onestep}
Let $(\M, X)$ be a type II$_1$ factorial tracially complete C$^*$-algebra with CPoU. Let $a \in \M_+$. Let $f:X \rightarrow [0,1]$ be a continuous affine function such that $f(\tau) \leq d_\tau(a)$ for all $\tau \in X$. Then for any $\eps > 0$, there exists a positive contraction $p_1 \in \overline{aMa}$ such that
\begin{align}
	\|p_1^2 - p_1\|_{2,X} < \eps, \label{onestep:cond1}\\
	\sup_{\tau \in X}|\tau(p_1) - f(\tau)| < \eps. \label{onestep:cond2}
\end{align}
Moreover, suppose we are also given $\eps_0 \in (0,\tfrac{1}{4})$ and a positive contraction $p_0 \in \overline{aMa}$ such that
\begin{align}
	\|p_0^2 - p_0\|_{2,X} < \eps_0, \label{moreover:cond1}\\
	\sup_{\tau \in X}|\tau(p_0) - f(\tau)| < \eps_0. \label{moreover:cond2}
\end{align}
Then $p_1 \in \overline{aMa}$ can be chosen such that, in addition to \eqref{onestep:cond1} and \eqref{onestep:cond2}, we have
\begin{equation}
	\|p_1 - p_0\|_{2,X} < 8\eps_0^{1/4}. \label{onestep:cond3}
\end{equation}
\end{lemma}
\begin{proof}
We first prove the lemma including the ``moreover'' part.

By Proposition \ref{prop:extending-affine}, we can extend $f$ to a continuous affine function on $T(\M)$. Then by Proposition \ref{prop:CP}, there is a self-adjoint $c \in \M$ such that $\tau(c) = f(\tau)$ for all $\tau \in X$ (in fact for all $\tau \in T(\M))$.   

Let $\tau \in X$. As $\M$ has type II$_1$, $\pi_\tau(\M)''$ is type II$_1$ von Neumann algebra. Let $s_\tau$ denote the support projection of $\pi_\tau(a)$. 
Let $\sigma \in T(\pi_\tau(\M)'')$. As $X$ is a closed face, $\sigma \circ \pi_{\tau} \in X$ by \cite[Lemma 2.10]{TraciallyComplete}.
Assuming $\sigma$ is normal, we have $\sigma(s_\tau) = d_\sigma(\pi_\tau(a)) = d_{\sigma \circ \pi_\tau}(a)$.
Therefore, we have $\sigma(\pi_\tau(c)) = f(\sigma \circ \pi_\tau) \leq \sigma(s_\tau)$ for all normal $\sigma \in T(\pi_\tau(\M)'')$.
As normal traces are dense, we have 
\begin{equation}
    \sigma(\pi_\tau(c)) \leq \sigma(s_\tau) \label{onestep:trace-estimate} 
\end{equation}
for all $\sigma \in T(\pi_\tau(\M)'')$.

By Proposition~\ref{prop:close-projection}, there is a projection $\bar{p}_\tau \in s_\tau\pi_\tau(\M)''s_\tau$ such that $\|\bar{p}_\tau - \pi_\tau(p_0)\|_{2,\sigma} < 4\eps_0^{1/4}$ and $\sigma(\bar{p}_\tau) = \sigma(\pi_\tau(c))$ for all $\sigma \in T(\pi_\tau(\M)'')$. Since $\bar{p}_\tau-\pi_\tau(c)$ vanishes on all traces, it can be written as a sum of at most 10 commutators by \cite[Th{\'e}or{\`e}me~2.3]{FH80}. 
Since $\pi_\tau(\overline{a \M a})$ is dense in $s_\tau\pi_\tau(\M)''s_\tau$ with respect to the ultrastrong topology, by Kaplansky's Density Theorem, there are a positive contraction $p_\tau \in \overline{a \M a}$ and elements $x_{j, \tau}, y_{j, \tau} \in \M$ for $j \in \{1, \ldots, 10\}$ such that
\begin{align}
    \Big\| p_\tau - c - \sum_{i=1}^{10} [ x_{j, \tau}, y_{j, \tau} ] \Big\|_{2, \tau} &< \tfrac{1}{2}\eps, \label{onestep:fibrewise-1}\\ 
    \|p_\tau - p_\tau^2\|_{2,\tau} &< \tfrac{1}{2}\eps,  \label{onestep:fibrewise-2}\\
    \|p_\tau - p_0\|_{2,\tau} &< 4\eps_0^{1/4}.  \label{onestep:fibrewise-3}
\end{align}

For each $\tau \in X$, define
\begin{equation}
    a_\tau = |p_\tau - p_\tau^2|^2 + \Big| p_\tau - c - \sum_{j=1}^{10} [ x_{j, \tau}, y_{j, \tau} ] \Big|^2 + \frac{\eps^2}{64\eps_0^{1/2}}|p_\tau - p_0|^2
\end{equation}
and note that $\tau(a_\tau) < \tfrac{3}{4}\eps^2$. Since the topology on $X$ is inherited from the weak$^*$ topology on $T(\M)$, the map $\rho \mapsto \rho(a_\tau)$ is continuous on $X$. Hence, there is a open neighbourhood $U_\tau$ of $\tau$ in $X$ such that $\rho(a_\tau) < \tfrac{3}{4}\eps^2$ for all $\rho \in U_\tau$.  

The collection $\{U_\tau: \tau \in X\}$ forms an open cover of $X$. By compactness of $X$, 
there is a finite subcover $\{U_{\tau_1},\ldots,U_{\tau_k}\}$.
Therefore, there are $\tau_1, \ldots, \tau_k \in X$ such that
    \begin{equation}
         \sup_{\tau \in X} \min_{1 \leq i \leq k} \tau(a_{\tau_i}) < \tfrac{3}{4}\eps^2.
    \end{equation}
Let $S \subseteq \M$ be the $\|\cdot\|_{2,X}$-separable subset generated by $p_{\tau_i}$, $x_{j, \tau_i}$ and $y_{j, \tau_i}$ for $i=1,\ldots,k$ and $j=1,\ldots,10$. Let $\omega \in \beta\N \setminus \N$ be a free ultrafilter. By CPoU, there are projections $q_1, \ldots, q_k \in \M^\omega \cap S'$ summing to $1_{\M^\omega}$ such that
\begin{equation}\label{CPoU:output}
    \tau(q_i a_{\tau_i}) \leq  \tfrac{3}{4}\eps^2 \tau(q_i), \qquad  \tau \in X^\omega, i = 1, \ldots, k.
\end{equation}
Define 
\begin{equation}
     p = \sum_{i=1}^k p_{\tau_i}^{1/2} q_i p_{\tau_i}^{1/2}, \quad x_j = \sum_{i=1}^k q_i x_{j, \tau_i}, \quad \text{and} \quad y_j = \sum_{i=1}^k q_i y_{j, \tau_i}
\end{equation}
for $j= 1, \ldots, 10$.  Then, by summing \eqref{CPoU:output} for $i \in \{1,\ldots,k\}$ and taking advantage of the orthogonality and commutativity properties of $q_1,\ldots,q_k$, we obtain 
    \begin{equation}
     \tau\left(|p - p^2|^2 + \Big| p - c - \sum_{j=1}^{10} [ x_{j}, y_{j} ] \Big|^2 + \frac{\eps^2}{64\eps_0^{1/2}}|p - p_0|^2\right) \leq \tfrac{3}{4}\eps^2 \label{CPoU:magic}
\end{equation}
    for all $\tau \in X^\omega$. By linearity, the left hand side of \eqref{CPoU:magic} is a sum of three positive terms. Therefore, we deduce that
\begin{align}
    \Big\| p - c - \sum_{i=1}^{10} [ x_{j}, y_{j} ] \Big\|_{2, \tau}^2 &\leq \tfrac{3}{4}\eps^2 < \eps^2, \\ 
    \|p - p^2\|_{2,\tau}^2  &\leq \tfrac{3}{4}\eps^2 < \eps^2, \\
    \|p - p_0\|_{2,\tau}^2 &\leq \frac{64\eps_0^{1/2}}{\eps^2}  \cdot \frac{3\eps^2}{4} < 64\eps_0^{1/2}
\end{align}
for all $\tau \in X^\omega$.   

Let $(q_i^{(n)})_{n=1}^\infty$ be a representative sequence for $q_i \in \M^\omega$. 
Let $p^{(n)} = \sum_{i=1}^k p_{\tau_i}^{1/2} q_i^{(n)} p_{\tau_i}^{1/2} \in \overline{a \M a}$.  Then $(p^{(n)})_{n=1}^\infty$ is a representative sequence for $p \in \M^\omega$. 
Let $(x^{(n)}_{j})_{n=1}^\infty$ and $(y^{(n)}_{j})_{n=1}^\infty$ be representative sequences for $x_j$ and $y_j$, respectively, for $j=1,\ldots,10$.
By the definition of a limit along an ultrafilter, there is a non-empty set of natural numbers $I \in \omega$ such that for any $n \in I$, we have 
\begin{align}
    \Big\| p^{(n)} - c - \sum_{i=1}^{10} [ x^{(n)}_{j}, y^{(n)}_{j} ] \Big\|_{2, X} &< \eps, \label{eqn:result-1} \\ 
    \|p^{(n)} - (p^{(n)})^2\|_{2,X}  &< \eps, \label{eqn:result-2}\\
    \|p^{(n)} - p_0\|_{2,X} &< 8\eps_0^{1/4}.\label{eqn:result-3}
\end{align} 
Choose the smallest $n \in I$ and set $p_1 = p^{(n)}$. Then \eqref{onestep:cond1} and \eqref{onestep:cond3} are satisfied. Since commutators have trace zero and $c$ was chosen with $\tau(c) = f(\tau)$ for all $\tau \in X$, \eqref{onestep:cond2} follows from \eqref{eqn:result-1}. This completes the proof of the lemma with the ``moreover'' part. 

In the absence of an approximate projection $p_0$ that we want to remain close to, the proof above simplifies. In each fibre, we just need to choose a projection $\bar{p}_\tau \in s_\tau\pi_\tau(\M)''s_\tau$ with $\sigma(\bar{p}_\tau) = \sigma(\pi_\tau(c))$ for all $\sigma \in T(\pi_\tau(\M)'')$, which exists by \eqref{onestep:trace-estimate} as $\M$ is type II$_1$. We can now follow the rest of the proof above with $p_0=0$ and $\eps_0 = 1$ (or adapt the CPoU argument to this simpler case).
\end{proof}

By repeated use of the previous lemma we can construct a Cauchy sequence of approximate projections, which will converge in the tracially complete C$^*$-algebra $(\M,X)$.    
\begin{theorem}\label{thm:projections}
Let $(\M, X)$ be a type II$_1$ factorial tracially complete C$^*$-algebra with CPoU. Let $a \in \M_+$. Let $f:X \rightarrow [0,1]$ be a continuous affine function such that $f(\tau) \leq d_\tau(a)$ for all $\tau \in X$. 
There exists a projection $p \in \M$ such that
\begin{enumerate}[(i)]
 \item $\tau(p) = f(\tau)$ for all $\tau \in X$; \label{projections:cond1}
 \item for any $b \in M_+$ with $ba = ab = a$, we have $bp = pb = p$. \label{projections:cond2}
\end{enumerate}
\end{theorem}
\begin{proof}
Take $\eps_i = \tfrac{1}{16^{i+1}}$. Applying Lemma \ref{lem:onestep} inductively, we construct a sequence of positive contractions $(p_i)_{i=0}^\infty$ in $\overline{aMa}$ such that
\begin{align}
	\|p_i^2 - p_i\|_{2,X} &< \eps_i, &i \geq 0, \label{eqn:seq-1} \\
	\sup_{\tau \in X}|\tau(p_i) - f(\tau)| &< \eps_i, &i \geq 0, \label{eqn:seq-2}\\
	\|p_{i} - p_{i-1}\|_{2,X} &< 8\eps_{i-1}^{1/4}, &i \geq 1. \label{eqn:seq-3}
\end{align}

Since $\sum_{i=1}^\infty 8\eps_{i-1}^{1/4} < \infty$, the sequence $(p_i)_{i=0}^\infty$ is a $\|\cdot\|_{2, X}$-Cauchy sequence. As each $p_i$ is a positive contraction, we have $\|p_i\| \leq 1$ for all $i \in \N$.   
Therefore, as $\M$ is tracially complete, $(p_i)_{i=0}^\infty$ converges in $\|\cdot\|_{2, X}$-norm to some contraction $p \in \M$. 

By \eqref{eqn:seq-1} and the $\|\cdot\|_{2,X}$-continuity of multiplication on $\|\cdot\|$-bounded sets (see Section \ref{subsec:completion}), we obtain $p^2 = p$. As every $\tau \in X$ is $\|\cdot\|_{2, X}$-continuous, we obtain $\tau(p) = f(\tau)$ from \eqref{eqn:seq-2} for every $\tau \in X$. 
Suppose $b \in \M_+$ satisfies $ba =ab = a$. Then $bp_i = p_i$ for all $i \in \N$ as $p_i \in \overline{aMa}$. Since left multiplication by $b$ is $\|\cdot\|_{2, X}$-continuous on $\M$, we have $bp = p$. Similarly, we have $pb = p$.
\end{proof}

The projection $p$ constructed in Theorem \ref{thm:projections} lies in the $\|\cdot\|_{2,X}$-closure of the hereditary subalgebra $\overline{a\M a}$.
In general, it need not lie in $\overline{a \M a}$, as the following example shows.
\begin{example}
Consider the case where $\M$ is a II$_1$ factor, $X = \{\tr_\M\}$, and $a$ is a positive element with $d_\tr(a) = 1$.
The only projection $p \in \M$ with $\tr(p)=1$ is the identity. 
However, if $1_M \in \overline{a\M a}$, then $a^{1/n}$ converges in norm to the identity, which is not true when $a$ has spectrum $[0,1]$ for example. 
\end{example}
Nevertheless, Theorem \ref{thm:projections} \eqref{projections:cond2} allows us to control the Cuntz equivalence class of the projection $p$ to some degree. For the remainder of this section, we introduce the following notation.

\begin{notation}
Let $a,b \in A$ be positive elements of the C$^*$-algebra $A$. We write $b \triangleright a$ when $b$ acts as a unit on $a$, i.e.\ when $ba=ab=a$. 
Given $0 < \eps_1 < \eps_2 < 1$, we write $\eta_{\eps_1,\eps_2}:[0,1]\rightarrow[0,1]$ for the continuous function that is zero on $[0,\eps_1]$, affine on $[\eps_1,\eps_2]$, and one on $[\eps_2,1]$.
\end{notation}

We first consider upper bounds on the Cuntz equivalence class of the projections $p$ constructed by Theorem \ref{thm:projections}.

\begin{proposition}\label{prop:projections-below}
Let $(\M, X)$ be a type II$_1$ factorial tracially complete C$^*$-algebra with CPoU. 
Let $f:X \rightarrow [0,1]$ be a continuous affine function.
Let $a \in \M_{+,1}$ and $\eps > 0$. 
Suppose that 
\begin{equation}f(\tau) \leq d_\tau((a-\eps)_+)\end{equation} for all $\tau \in X$.
Then there exists a projection $p \in \overline{a\M a}$ with $\tau(p) = f(\tau)$ for all $\tau \in X$. In particular, $p \precsim a$. 
\end{proposition}
\begin{proof}
We have $\eta_{0,\eps}(a)  \triangleright (a-\eps)_+$. 
Therefore, by Theorem \ref{thm:projections}, there exists a projection $p \in \M$ such $\tau(p) = f(\tau)$ for all $\tau \in X$ and $\eta_{0,\eps}(a) \triangleright p$.  
Since $\eta_{0,\eps}(a) \triangleright p$, we have $p \in \overline{a\M a}$ and $p \precsim a$. 
\end{proof}

We now consider lower bounds on the Cuntz equivalence class of the projections constructed by Theorem \ref{thm:projections}.
\begin{proposition}\label{prop:projections-above}
Let $(\M, X)$ be a type II$_1$ factorial tracially complete C$^*$-algebra with CPoU. Let $a \in \M_{+,1}$. 
For any continuous affine function $g:X \rightarrow [0,1]$ with $d_\tau(a) \leq g(\tau)$ for all $\tau \in X$. 
There exists a projection $p \in \M$ with $\tau(p) = g(\tau)$ for all $\tau \in X$ and $a \precsim p$. 
\end{proposition}
\begin{proof}
Let $\eps > 0$. 
Then $\tau(\eta_{0,\eps}(a)) \leq d_\tau(a) \leq g(\tau)$ for all $\tau \in X$. 
Hence, $d_\tau(1 - \eta_{0,\eps}(a)) \geq \tau(1 - \eta_{0,\eps}(a)) \geq 1-g(\tau)$. 
Moreover, since $\eta_{0,\eps}(a) \triangleright \eta_{\eps,2\eps}(a)$, we have $1 - \eta_{\eps,2\eps}(a) \triangleright 1 - \eta_{0,\eps}(a)$.

By Theorem \ref{thm:projections}, there exists a projection $q \in \M$ with $\tau(q) = 1 - g(\tau)$ for all $\tau \in X$ and $1 - \eta_{\eps,2\eps}(a) \triangleright q$. 
Set $p = 1-q$. Then $\tau(p) = g(\tau)$ for all $\tau \in X$ and $p \triangleright \eta_{\eps,2\eps}(a)$.
It follows that $(a-2\eps)_+ \precsim \eta_{\eps,2\eps}(a) \precsim p$.

The projection $p$ constructed above may depend on $\eps$. However, by \cite[Theorem 7.17]{TraciallyComplete} the Murray--von Neumann equivalence class of $p$ does not, as it is completely determined by tracial data. Hence, we have $(a-2\eps)_+ \precsim p$ for all $\eps > 0$, which implies $a \precsim p$.
\end{proof}

Combining Proposition \ref{prop:projections-above} and Proposition \ref{prop:projections-below}, we obtain the following theorem.
\begin{theorem}\label{thm:Cuntz-dense}
	Let $(\M, X)$ be a type II$_1$ factorial tracially complete C$^*$-algebra with CPoU. Let $a \in \M_{+,1}$. 
	For any $\eps > 0$, there exists a projection $p \in \M$ with $(a-\eps)_+ \precsim p \precsim a$.
\end{theorem}
\begin{proof}
	Define a continuous affine function $f:X \rightarrow [0,1]$ by $f(\tau) = \tau(\eta_{\eps/2, \eps}(a))$. Then
	\begin{equation}
		d_\tau((a-\eps)_+) \leq f(\tau) \leq d_\tau((a-\eps/2)_+).
	\end{equation}
	By Proposition \ref{prop:projections-below}, there exists a projection $p \in \M$ with $\tau(p) = f(\tau)$ for all $\tau \in X$ and $p \precsim a$.
	By Proposition \ref{prop:projections-above}, there exists a projection $q \in \M$ with $\tau(q) = f(\tau)$ for all $\tau \in X$ and $(a-\eps)_+ \precsim q$.
	By \cite[Theorem 7.17(ii)]{TraciallyComplete}, $p$ and $q$ are unitary equivalent. Hence, $(a-\eps)_+ \precsim p \precsim a$. 
\end{proof}
The conclusion of Theorem~\ref{thm:Cuntz-dense} could be interpreted informally as saying that $(\M,X)$ has real rank zero at the level of the Cuntz semigroup. 
Indeed, if $A$ is a C$^*$-algebra with real rank zero and $a \in A_+$, then the hereditary subalgebra $\overline{aAa}$ has an approximate unit consisting of projections, so there exists a projection $p \in A$ with $(a-\eps)_+ \precsim p \precsim a$ by \cite[Proposition 2.2]{Ro92}. 
Passing to the stabilisation, this argument shows that the Cuntz semigroup $\mathrm{Cu}(A)$ of a C$^*$-algebra with real rank zero is \emph{algebraic} in the sense of \cite[Section 5.5]{RFT18}, i.e.\ every element of $\mathrm{Cu}(A)$ is the supremum of a sequence of compact elements. 
This observation goes back to \cite[Corollary 5]{CEI08}, where it is also shown that the converse holds when $A$ has stable rank one. In this language, Theorem~\ref{thm:Cuntz-dense} shows that $\mathrm{Cu}(\M)$ is algebraic. 

\section{The main theorems}

We now have everything we need to prove the main theorems of the paper. We first establish Theorem~\ref{thm:main-strict-comparison} from the introduction. 

\begin{theorem}[Theorem~\ref{thm:main-strict-comparison}]
    Let $(\M, X)$ be a type II$_1$ factorial tracially complete C$^*$-algebra with complemented partitions of unity. Then $\M$ has strict comparison with respect to the traces in $X$.
\end{theorem}
\begin{proof}
	Let $(\M, X)$ be a type II$_1$ factorial tracially complete C$^*$-algebra with CPoU.
    By the definition of a tracially complete C$^*$-algebra, $X$ is a closed subset of $T(\M)$. 
    Hence, $X$ is a compact subset of $\QT(\M)$, so strict comparison with respect to $X$ is well defined (see Section \ref{subsec:strict-comparison}).
        
	By Lemma \ref{lem:only-matrix}, it suffices to prove strict comparison with respect to $X$ for elements  $a, b \in (\M \otimes M_n)_+$ for all $n \in N$. 
	Moreover, since $(\M \otimes M_n, X)$ is also a type II$_1$ tracially complete C$^*$-algebra with CPoU by \cite[Proposition 6.14]{TraciallyComplete}, 
        it suffices to prove strict comparison with respect to $X$ for elements  $a, b \in \M$.
	
	Let $a,b \in \M$ and $\gamma > 0$. Suppose that $d_\tau(a) \leq (1-\gamma)d_\tau(b)$ for all $\tau \in K$.
	Let $\eps > 0$. By Lemma \ref{lem:cut-down}, there is $\delta > 0$ such that $d_\tau((a-\eps)_+) \leq (1-\frac{\gamma}{2})d_\tau((b-\delta)_+)$ for all $\tau \in K$.
	
	By Theorem \ref{thm:Cuntz-dense}, there are projections $p, q \in \M $ such that $(a-2\eps)_+ \precsim p \precsim (a-\eps)_+$ and $(b-\delta)_+ \precsim q \precsim b$.
	We have 
        \begin{equation}
            \tau(p) \leq d_\tau((a-\eps)_+) \leq \left(1-\frac{\gamma}{2}\right)d_\tau((b-\delta)_+) \leq \tau(q)
        \end{equation} for all $\tau \in K$.
	Hence, $p$ is Murray--von Neumann subequivalent to $q$ by \cite[Theorem 7.17(i)]{TraciallyComplete}.
	Therefore,  $(a-2\eps)_+ \precsim p \precsim q \precsim b$.
	Since $\eps$ was arbitrary, $a \precsim b$.
\end{proof}

Next, we establish Theorem~\ref{thm:main-TC} of the introduction. We isolate the Choquet theory and the application of Lemma \ref{lem:K-to-all} (due to Ng and Roberts \cite{NgR16}) in the following proposition, as it may be of independent interest.  We also observe that all quasitraces are traces in this case.

\begin{proposition}\label{prop:main-choquet-argument}
    Let $(\M, X)$ be a type II$_1$ factorial tracially complete C$^*$-algebra. Suppose $\M$ has strict comparison with respect to the traces in $X$. Then $T(\M) = X$. Moreover, all lower semi-continuous 2-quasitraces on $\M$ are additive. 
\end{proposition}
\begin{proof}
	Suppose there exists some $\sigma \in T(\M) \setminus X$.
	As $X$ is a closed face in the Choquet simplex $T(\M)$, there is a continuous affine function $f:T(\M) \rightarrow [0,1]$ such that $f|_X = 0$ but $f(\sigma) > 0$ by Proposition \ref{prop:exposed}.
	
    Let $\eps > 0$.
	By Proposition \ref{prop:CP}, there exists $a \in \M_+$ such that $\tau(a) = f(\tau) + \eps$ for all $\tau \in T(\M)$.
	Hence, $\tau(a) = \epsilon$ for all $\tau \in X$, as $f|_X = 0$.
	In particular, we have $\tau(a) \leq \tau(\eps 1_\M)$ for all $\tau \in X$.
	By Theorem \ref{thm:main-strict-comparison}, $\M$ has strict comparison with respect to $X$.
	Hence, by Lemma \ref{lem:K-to-all}, $\tau(a) \leq \tau(\eps 1_M)$ for all $\tau \in \QT(\M)$.
	Since $\sigma \in T(\M)$, we have $f(\sigma) + \epsilon = \sigma(a) \leq \sigma(\eps 1_M) = \eps$.
	Thus $f(\sigma) \leq 0$, and we have arrived at a contradiction. Therefore, $T(\M) = X$.

    Since $\M$ has strict comparison with respect to $X$, it also has strict comparison with respect to the (larger) set of all lower semi-continuous extended traces on $\M$. Hence, all lower semi-continuous 2-quasitraces on $\M$ are additive by \cite[Theorem 3.6(i)]{NgR16}. 
\end{proof}

\begin{corollary}[Theorem~\ref{thm:main-TC}]
    Let $(\M, X)$ be a type II$_1$ factorial tracially complete C$^*$-algebra with complemented partitions of unity. Then $T(\M) = X$.
\end{corollary}
\begin{proof}
    This is now an immediate consequence of Proposition~\ref{prop:main-choquet-argument} and Theorem~\ref{thm:main-strict-comparison}.
\end{proof}

We now specialise Theorem~\ref{thm:main-TC} to the case where $(\M,X)$ is a trivial W$^*$-bundle. This will prove Corollary~\ref{cor:main-bundle} from the introduction as a special case.  

\begin{proposition}\label{prop:bundles-gamma}
    Let $K$ be a compact Hausdorff space and $N$ be a II$_1$ factor. Let $C_\sigma(K, N)$ be the trivial W$^*$-bundle over $K$ with fibre $N$.
    Suppose either $K$ is totally disconnected or $N$ has property $\Gamma$.
    Then every trace $\tau \in T(C_\sigma(K, N))$ is of the form
    \begin{equation}\label{eqn:main-bundle2}
        \tau(f) = \int_K \mathrm{tr}_N(f(x)) \, d\mu(x), \quad f \in C_\sigma(K, N),
    \end{equation}
    for some Radon probability measure $\mu \in \mathrm{Prob}(K)$.
\end{proposition}
\begin{proof}
Set $\M = C_\sigma(K, N)$. 
Let $X \subseteq T(\M)$ be the set of all traces of the form \eqref{eqn:main-bundle2}. 
Then $(\M,X)$ is a tracially complete C$^*$-algebra by \cite[Proposition 3.6]{TraciallyComplete} and is factorial by \cite[Proposition 3.14]{TraciallyComplete} since $\pi_\tau(\M)'' \cong N$ is a factor for each $\tau \in \partial_e X$. 

The C$^*$-algebra $\M$ contains a unital copy of $N$ (namely the constant functions), which in turn contains a unital copy of the matrix algebra $M_n$ for every $n \in \N$ (as $N$ is a II$_1$ factor). Therefore, the type I part of $\pi_\tau(\M)''$ must be zero for every $\tau \in X$. Hence, $(\M,X)$ is type II$_1$. 

Suppose the II$_1$ factor $N$ has property $\Gamma$. Then the tracially complete C$^*$-algebra $(\M,X)$ has property $\Gamma$. Indeed, an approximately central net of projections of trace $\tfrac{1}{2}$ in $N$, when viewed as a constant functions, will be an approximately central net of projection of uniform trace $\tfrac{1}{2}$ in $\M$. Therefore, $(\M,X)$ has CPoU by \cite[Theorem 1.4]{TraciallyComplete}. Suppose instead $K$ is totally disconnected. Then since $\partial_e X \cong K$, we see that $(\M,X)$ has CPoU by \cite[Proposition~6.3]{TraciallyComplete}.

In both cases, Theorem~\ref{thm:main-TC} applies and we have $T(\M) = X$.  
\end{proof}
\begin{corollary}[Corollary~\ref{cor:main-bundle}]
    Every trace on the trivial W$^*$-bundle $C_\sigma(K, \RR)$ is of the form \eqref{eqn:main-bundle1}.
\end{corollary}
\begin{proof}
    This is now an immediate consequence of Proposition~\ref{prop:bundles-gamma} since $\RR$ has property $\Gamma$.
\end{proof}

Finally, we specialise Theorem~\ref{thm:main-TC} to the case of uniform tracial completions of C$^*$-algebras to prove Theorem \ref{thm:main} from the introduction.
\begin{theorem}[Theorem~\ref{thm:main}]
    Let $A$ be a C$^*$-algebra with $T(A)$ compact and non-empty. Suppose $A$ absorbs the Jiang--Su algebra $\Z$ tensorially, i.e.\ $A \otimes \Z \cong A$. Then the trace problem has a positive solution. 
\end{theorem}
\begin{proof}
     Let $A$ be a $\Z$-stable C$^*$-algebra. Set $\M = \completion{A}{T(A)}$ and let $X$ be the set of uniform 2-norm continuous extensions of traces in $T(A)$. 
     By \cite[Proposition 3.23]{TraciallyComplete}, $(\M,X)$ is a factorial tracially complete C$^*$-algebra.
     
     Since $A$ is $\Z$-stable, the tracially complete C$^*$-algebra $(\M,X)$ has the McDuff property by \cite[Proposition 5.17]{TraciallyComplete}, so $(\M,X)$ has property $\Gamma$ by \cite[Proposition 5.22]{TraciallyComplete}. Therefore, by \cite[Theorem 1.4]{TraciallyComplete}, $(\M,X)$ has CPoU. 
     Since $(\M,X)$ has the McDuff property, the type I part of $\pi_\tau(\M)''$ must vanish for every $\tau \in X$. Therefore, $(\M,X)$ is type II$_1$.
     
     As the conditions of Theorem \ref{thm:main-TC} are satisfied,  $T(\M) = X$.
\end{proof}

\end{document}